\documentclass{winnower}

\begin{document}

\definecolor{uuuuuu}{rgb}{0.26666666666666666,0.26666666666666666,0.26666666666666666}
\definecolor{xdxdff}{rgb}{0.49019607843137253,0.49019607843137253,1}
\definecolor{qqttcc}{rgb}{0,0.2,0.8}

\title{A Group Theory Proof of Pascal's Theorem}

\author{Kaylee Wiese}

\date{}

\newcommand{\X}{X} 
\newcommand{\id}{o} 
\newcommand{\Lref}{L} 
\newcommand{\XL}{X_L} 

\newcommand{\aP}{a} 
\newcommand{\bP}{b} 
\newcommand{\cP}{c} 
\newcommand{\dP}{d} 
\newcommand{\eP}{e} 
\newcommand{\fP}{f} 

\newcommand{\pref}{p} 
\newcommand{\newPt}{g} 

\newcommand{\Pp}{A_1} 
\newcommand{\PP}{A_2} 
\newcommand{\Q}{B_1} 
\newcommand{\QQ}{B_2} 
\newcommand{\R}{C_1} 
\newcommand{\RR}{C_2} 

\newcommand{\Ppt}{p} 
\newcommand{\Qpt}{q} 
\newcommand{\Rpt}{r} 

\newcommand{\ALine}{P} 
\newcommand{\BLine}{Q} 
\newcommand{\CLine}{R} 

\newcommand{\T}{T} 

\newcommand{\Par}{\mathcal{P}}
\newcommand{\Hyp}{H}
\newcommand{\Circ}{S^1}
\newcommand{\Linfty}{L_{\infty}}

\newcommand{\mycomment}[1]{}

\maketitle

\begin{abstract}
It will be shown that Pascal's Theorem is equivalent to the associativity of a natural binary operation on conic sections. A novel proof for Pascal's Theorem will then be given by showing that this binary operation is associative independent of Pascals Theorem. Specifically, this operation is equivalent to either addition of real numbers, multiplication of real numbers, or rotations on the plane depending on the type of the conic. Since each of these is already known to be associative, it will follow that the binary operation is associative and this will prove Pascal's Theorem.
\end{abstract}

\section{Introduction}
Pascal's Theorem is an important result in planar geometry and has attracted many proofs. These range from pure geometry, applications of Bezout's Theorem, and hyperbolic geometry \hyperlink{https://www.heldermann-verlag.de/jgg/jgg01_05/jgg0101.pdf}{\citep{Author3Year3}}, \hyperlink{https://www.jstor.org/stable/pdf/2315885.pdf}{\citep{Author4Year4}}, \hyperlink{https://arxiv.org/pdf/2012.14883.pdf}{\citep{Author5Year5}}, \hyperlink{https://arxiv.org/pdf/1903.00460.pdf}{\citep{Author6Year6}}, \hyperlink{https://www.jstor.org/stable/pdf/2688012.pdf}{\citep{Author7Year7}}. In this article we will continue this tradition by providing a proof of Pascal's Theorem based on group theory and the arithmetic properties of elementary operations. Pascal's Theorem is stated as follows:

\begin{theorem} \textbf{(Pascal's Theorem)}
Let $\X$ be a conic section and let $\aP,\bP,\cP,\dP,\eP,\fP$ be six points on $\X$. Connect consecutive pairs of points on this list with the lines 
\begin{eqnarray*}
\Pp = \overline{\aP\bP} && \PP=\overline{\dP\eP}\\
\Q = \overline{\bP\cP} && \QQ=\overline{\eP\fP}\\
\R = \overline{\cP\dP} && \RR = \overline{\fP\aP}
\end{eqnarray*}
Then the points of intersection, $\Pp\cap \PP$, $\Q\cap \QQ$, $\R\cap \RR$ are colinear on the projective plane.
\end{theorem}
In this theorem we can assume that some of the points are identical, drawing tangent lines when necessary, as long as the resulting pairs of lines, such as $\Pp$ and $\PP$, are not identical. The main idea of the group theory proof comes from the fact that a conic section with a marked point has a natural group structure \hyperlink{https://arxiv.org/abs/math/0311306}{\citep{AuthorYear}}. This group operation is purely geometric and can be thought of as a conic version of the elliptic curve group operation. On the affine plane the group structure structure is given as follows:

\vspace{.25cm}\noindent\textbf{Binary Operation on $\X$:}  Let $\X$ be a conic, $\id$ a fixed point on $\X$, and $\aP,\bP$  two points on $\X$. Draw the line parallel to $\overline{\aP\bP}$ which passes through $\id$. This line will intersect $\X$ at two points including $\id$, call the second point $\aP\oplus \bP$. 
\vspace{.25cm}

If $\aP=\bP$, we understand $\overline{\aP\bP}$ to be the line at this point which is tangent to the curve. If the resulting line is tangent at $\id$, then we take $\aP\oplus \bP = \id$. This binary operation turns $\X$ into a group with identity $\id$. This group has been studied as a way to understand elliptic curves in the context of the Birch and Swinnerton-Dyer Conjecture \hyperlink{https://arxiv.org/abs/math/0311306}{\citep{AuthorYear}}, and its properties have been explored in depth as well \hyperlink{https://www.jstor.org/stable/pdf/40378668.pdf}{ \citep{Author2Year2}}. The main observation of this article is that Pascal's Theorem is required to prove the associativity of this operation. A natural question to ask is if this implication works in reverse: Does Pascal's Theorem follow from the associativity of this binary operation? If so, can we prove associativity independent of Pascal's theorem in order to obtain a new proof of Pascal's Theorem?

In this paper, we will answer this question with an affirmative. In the usual proof of associativity for this operation the projective case of Pascal's Theorem is used and the line produced by Pascal's Theorem is the line at infinity, regardless of the conic we work with. This is because of the use of parallel lines in the definition of the conic group operation. This does not cover all cases of Pascal's Theorem, so we will need to generalize our construction. This is done in the first section and with this we can prove that Pascal's Theorem is equivalent to the associativity of this binary operation on any conic. In the following section this binary operation will be computed explicitly in a few special cases, resulting in real number addition, real number multiplication, and planar rotation, each of which is already known to be associative. Transformations of the projective plane will provide isomorphisms between any conic and one of these special cases. Associativity is then proved independent of Pascal's Theorem through these isomorphisms, which then provides a proof of Pascal's Theorem itself.
\section{Conic Addition and Pascal's Theorem}

In this section we generalize the binary operation and expand on the results of \hyperlink{https://www.jstor.org/stable/pdf/40378668.pdf}{ \citep{Author2Year2}}. The main difference is that throughout the article we will be working on the real projective plane rather than the affine plane.  In the sequel, if $\aP,\bP$ are two points on a conic and $\aP=\bP$, then we assume that the line $\overline{\aP\bP}$ is the tangent line of the conic at this point.

Let $\X$ be a conic and $\Lref$ a line on the real projective plane, and let $\id$ be a point on $\X$ which is not on $\Lref$. We will call such a triple, $(\X,\Lref,\id)$ a \textbf{Marked Conic}. Each marked conic gives rise to a natural group structure on $\XL:=X-L$ as follows (see Figure 1 for an illustration).

\vspace{.25cm}\noindent\textbf{Binary Operation on $\XL$:} Let $\aP,\bP\in \XL$ and let $\overline{\aP\bP}$ be the line through $\aP$ and $\bP$.  This line will intersect $\Lref$ at some point $\pref$ on the projective plane not on $X$. The line $\overline{\id\pref}$ will intersect the conic twice, at least once at $\id$ and then another point. Define  $\aP\oplus \bP$ to be this second point.
\vspace{.25cm}

\noindent It should be noted that if $\Lref$ is the line at infinity, then this operation reproduces the operation originally described in the introduction.

\begin{figure}[!h]
\centering
\scalebox{0.55}{
\begin{tikzpicture}[line cap=round,line join=round,>=triangle 45,x=1cm,y=1cm]
\clip(-8.56928333823871,-7.257432724477938) rectangle (9.410011990495365,4.723067299706999);
\draw [rotate around={-32.650802296595565:(-2.31,-1.81)},line width=2pt] (-2.31,-1.81) ellipse (4.4906330669392664cm and 2.587911386019304cm);
\draw [line width=2pt,domain=-8.56928333823871:9.410011990495365] plot(\x,{(--12.8484-4.8*\x)/1.08});
\draw [line width=0.8pt,domain=-8.56928333823871:9.410011990495365] plot(\x,{(--4.61151686043069--1.4170040411949352*\x)/5.0308424088331325});
\draw [line width=0.8pt,domain=-8.56928333823871:9.410011990495365] plot(\x,{(--11.685368592697273-6.616605331960558*\x)/-2.3467314496867435});
\draw (-6.826043677087688,-0.3357458346136021) node[anchor=north west] {$\aP$};
\draw (-1.288694165196205,1.168225637751982) node[anchor=north west] {$\bP$};
\draw (1.6167052700554991,-1.4637244388877901) node[anchor=north west] {$\aP\oplus\bP$};
\draw (0.12982438260315643,-5.274924874541486) node[anchor=north west] {$\id$};
\begin{scriptsize}
\draw[color=black] (-2.7755750526485476,1.6211261379529818) node {$\X$};
\draw[color=black] (1.4628900058362913,4.253076214592753) node {$\Lref$};
\draw [fill=xdxdff] (-6.2714319164077414,-0.8497836270306787) circle (3.5pt);
\draw [fill=xdxdff] (-1.2405895075746085,0.5672204141642565) circle (3.5pt);
\draw[color=black] (-7.902750526622143,-1.5919038257371296) node {$\overline{\aP\bP}$};
\draw [fill=uuuuuu] (2.3232685439939127,1.571028693360387) circle (2pt);
\draw[color=uuuuuu] (2.539596855370746,1.4160391189940384) node {$\pref$};
\draw [fill=black] (-0.023462905692830827,-5.045576638600171) circle (3.5pt);
\draw[color=black] (3.5821225350787107,3.8087210068483763) node {$\overline{\pref\id}$};
\draw [fill=xdxdff] (1.1733169991261256,-1.6712577319926922) circle (3.5pt);
\end{scriptsize}
\end{tikzpicture}
}
\caption{Addition Rule on Marked Conics}
\end{figure}
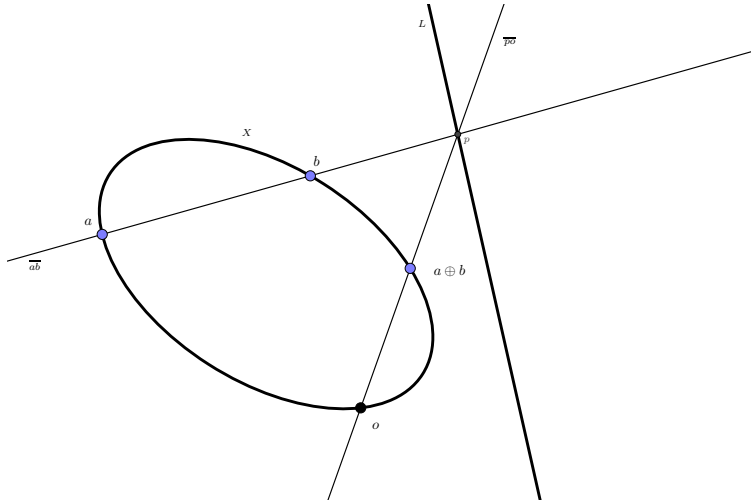

\begin{prop} The binary operation on $\XL$ induced by the marked conic $(\X,\Lref,\id)$ is well-defined, commutative, $\id$ is an identity element, and inverses exist.
\end{prop}
\begin{proof}
We must first show that this operation is well-defined. For each $\aP,\bP\in \XL$, the line $\overline{\aP\bP}$ clearly exists. This line cannot be identical to $\Lref$, since $\aP$ and $\bP$ are not on $\Lref$. Furthermore, the point where $\Lref$ and $\overline{\aP\bP}$ intersect cannot be on $\X$ because $\overline{\aP\bP}$ can only intersect $\X$ twice, and it already intersects at $\aP$ and $\bP$ which are not on $\Lref$. So the point $\pref$, intersecting $\Lref$ and $\overline{\aP\bP}$, exists and is not on $\X$. The rest of the construction is clear.

The operation is commutative since we have $\overline{\aP\bP}=\overline{\bP\aP}$. The identity is $\id$ since if $\bP=\id$, then we will get $\overline{\aP\bP}=\overline{\aP\id}=\overline{\id\pref}$. Finally, to find inverses, first find the line tangent to $\X$ at $\id$, and let $\pref_{\id}$ be the intersection point of this line with $\Lref$. To get the inverse of $\aP$, construct the line $\overline{\aP\pref_{\id}}$, and the inverse will be the second point of intersection of this line with $\X$.
\end{proof}

The main result relies on the following equivalence

\begin{theorem} Pascal's Theorem is equivalent to the statement that for every marked conic $(\X,\Lref,\id)$, the induced binary operation on $\XL$ is associative
\end{theorem}

\begin{proof}
\noindent\textbf{Pascal's Theorem implies Associativity:} Let $\aP,\bP,\cP$ be three points on $\XL$. Note that if $\bP=\id$ then $\aP\oplus(\id\oplus \cP)=(\aP\oplus \id)\oplus \cP$ by direct computation, so we can assume $\bP\ne \id$ We can then get a list of six points on $\X$:
$$
\aP,\;\; \bP, \;\; \cP, \;\; \aP\oplus \bP, \;\; \id, \;\; \bP\oplus \cP
$$
And we have the six lines
\begin{eqnarray*}
\Pp = \overline{\aP\bP} && \PP=\overline{(\aP\oplus \bP)\id}\\
\Q = \overline{\bP\cP} && \QQ=\overline{\id(\bP\oplus \cP)}\\
\R = \overline{\cP(\aP\oplus \bP)} && \RR = \overline{(\bP\oplus \cP)\aP}
\end{eqnarray*}
Since $\bP\ne\id$, the pairs of lines such as $\Pp$ and $\PP$ are distinct. Assume that $\R$ and $\RR$ intersect $\Lref$ at the same point, $\pref$. Then second line in the construction of both $(\aP\oplus \bP)\oplus \cP$ and $\aP\oplus(\bP\oplus \cP)$ will be $\overline{\id\pref}$. Since they both share this line in common, these resulting points will be the same and associativity will follow.

We then just need to show that $\R$ and $\RR$ intersect $\Lref$ at the same point. By Pascal's Theorem the intersection points $\Pp\cap \PP$, $\Q\cap \QQ$, and $\R\cap \RR$ are all colinear. But $\Pp$ and $\PP$ are the two lines created in order to construct $\aP\oplus \bP$, and so they intersect on $\Lref$. Similarly, $\Q$ and $\QQ$ also intersect on $\Lref$. It then follows that the colinear points given by Pascal's Theorem are all on $\Lref$ and so $\R\cap \RR$ is on $\Lref$ and the result follows.

\vspace{0.25cm}\noindent\textbf{Associativity implies Pascal's Theorem:} Let $\X$ be a conic and let $\aP,\bP,\cP,\dP,\eP,\fP$ be six points on $\X$. We then have the six lines
\begin{eqnarray*}
\Pp = \overline{\aP\bP} && \PP= \overline{\dP\eP} \\
\Q = \overline{\bP\cP} && \QQ = \overline{\eP\fP} \\
\R = \overline{\cP\dP} && \RR = \overline{\fP\aP}
\end{eqnarray*}
These may be tangent lines to $\X$ if necessary, the only constraints on the points is that opposite lines, such as $\Pp$ and $\PP$, are not identical. Let $\Ppt,\Qpt,\Rpt$ be the intersection points in $\Pp\cap \PP$, $\Q\cap \QQ$, and $\R\cap\RR$ respectively. Let $\Lref$ be the line through $\Ppt$ and $\Qpt$. Pascal's Theorem will follow from proving that $\Rpt$ is on $\Lref$.

We will later show that we can assume that the already listed point $\eP$ is not on $\Lref$. This means that we can setup the marked conic $(\X,\Lref,\eP)$ which, by assumption, induces a group on $\XL$ through our geometric process. These six points will then have addition relations through this group operation as follows: Since $\Pp$, the line through $\aP$ and $\bP$, intersects $\Lref$ at the same point as $\PP$, a line which passes through the identity, we get that $\dP=\aP\oplus \bP$. Similarly, looking at $\Q$ and $\QQ$, we have $\fP=\bP\oplus \cP$. Therefore, from the assumption of associativity, we have
$$
\dP\oplus \cP = (a\oplus b)\oplus c = a\oplus (b\oplus c) = \aP\oplus \fP =:g
$$
Let $\newPt$ be this common point on $\X$. This equality says that $\R=\overline{\cP\dP}$, $\RR=\overline{\fP\aP}$ and $\overline{\eP\newPt}$ all pass through the same point on $\Lref$. This means that $\R$ and $\RR$ intersect on $\Lref$, or $\Rpt$ is on $\Lref$ which proves Pascal's Theorem.

So we just need to show that we can assume that $\eP$ is not on $\Lref$.

The main insight is that for any fixed six point, $a,b,c,d,e,f$, the geometric setup for Pascal's theorem is invariant under cycling this list. So even if $e$ is on the line $L$, we can cycle the list and try again. So we simply need to show that we can cycle such a list until the second-to-last point is not on the line in question.
Note that if any of the $\Ppt,\Qpt,\Rpt$ are equal, then Pascal's Theorem is trivially true, so we can assume they are all distinct. So let $\ALine,\BLine,\CLine$ be the three lines obtained from the points $\Ppt,\Qpt,\Rpt$.
$$
\ALine = \overline{\Ppt\Qpt}\;\;\;\;\;\;\;\;\BLine = \overline{\Qpt\Rpt}\;\;\;\;\;\;\;\;\CLine = \overline{\Rpt\Ppt}
$$
Observe that if we cycle the points $\aP\bP\cP\dP\eP\fP$ to $\bP\cP\dP\eP\fP\aP$, then this operation just cycles the pairs $(\Pp,\PP)$, $(\Q,\QQ)$, $(\R,\RR)$ which cycles through $\Ppt,\Qpt,\Rpt$ which then cycles through the three lines $\ALine,\BLine,\CLine$. We will show that we can cycle through points until the second-to-last element of the initial list list is not on the resulting line $\Lref$. For any second-to-last point on this list to not be on the resulting $\Lref$, which is $\ALine,\BLine,\CLine$ in turn through the cycle, then we just need to show that not all of these statements are simultaneously true:
\begin{eqnarray*}
\eP \in \ALine && \bP\in \ALine\\
\fP \in \BLine && \cP \in \BLine\\
\aP \in \CLine && \dP\in \CLine
\end{eqnarray*}
So assume they're all true. Let's take the first condition. If $\ALine$ does contain $\eP$, then we must have $\ALine=\PP=\QQ$. This is because
$$
\ALine = \overline{\eP\Ppt}=\overline{\eP\dP} = \PP
$$
and similarly for $\QQ$. Cycling the equation $\ALine=\PP$ through all six possibilities, we get:
\begin{eqnarray*}
    \ALine &=& \PP = \Pp  \\
    \BLine &=& \QQ = \Q  \\
    \CLine &=& \RR = \R
\end{eqnarray*}
This is a contradiction of our initial assumptions about these lines needing to be distinct. Therefore, at least one (an, indeed, many) of the above statements must be false, allowing us to cycle through the order of $\aP,\bP,\cP,\dP,\eP,\fP$ until we find an $\eP$ not on our choice of $\Lref$.

\end{proof}

We then have the concise result
\begin{cor}
Pascal's Theorem is equivalent to $(\XL,\oplus)$ being a group for all marked conics $(\X,\Lref,\id)$.
\end{cor}

\section{Marked Conics induce Groups}
We will then work to show that each $\XL$ is a group without appealing to Pascal's Theorem. The main strategy for this is to reduce each marked conic to a case that we can compute explicitly. The cases will be determined by how $\Lref$ intersects $\X$ - not at all, at a double point, or a two distinct points. These cases will correspond to a ellipse, parabola, and hyperbola respectively. We will then show that the binary operation on each ellipse is isomorphic to rotation, on each parabola is isomorphic to addition, and each hyperbola is isomorphic to multiplication - each of which is a group in their own right. This will then prove Pascal's Theorem through our equivalence.

The key tool is the fact that invertible projective transformations of the projective plane induce homomorphisms of these conic groups. It is important to note that we can talk about homomorphisms of binary operations alone, associativity is not needed. The important thing about projective trasformations is that they send lines to lines, conics to conics and, of course, intersections to intersections. Moreover, if $\Lref_1$ and $\Lref_2$ are any pairs of lines, then there are projective transformations which send $\Lref_1$ to $\Lref_2$. These properties immediately allow us to conclude:
\begin{prop}
Let $(\X_1,\Lref_1,\id_1)$ and $(\X_2,\Lref_2,\id_2)$ be two marked conics with associated binary operations $\oplus_1$ and $\oplus_2$ respectively. If $\T$ is a projective transformation so that $\T(\X_1)=\X_2$, $\T(\Lref_1)=\Lref_2$, and $\T(\id_1)=\id_2$, then for each $\aP,\bP\in (\X_1)_{\Lref_1}$ we have
$$
\T(\aP\oplus_1\bP) = \T(\aP)\oplus_2\T(\bP)
$$
\end{prop}
If we already know that the operation on $(\X_2,\Lref_2,\id_2)$ is associative and therefore induces a group structure on $(\X_2)_{\Lref_2}$, then this tells us that $(\X_1,\Lref_1,\id_1)$ also induces a group structure. With projective transformations, we can use explicit computations to prove that all marked conics induce a group without appealing to Pascal's Theorem. To this end, let $\Linfty$ be the line at infinity and define the following standard marked conics
\begin{enumerate}
    \item The \textbf{Parabola} is the marked conic $(\Par,\Linfty,(0,0))$ where $\Par$ is the parabola $y=x^2$.
    \item The \textbf{Hyperbola} is the marked conic $(\Hyp,\Linfty,(1,1))$ where $\Hyp$ is the hyperbola $xy=1$.
    \item The \textbf{Circle} is the marked conic $(\Circ,\Linfty,(1,0))$ where $\Circ$ is the unit circle $x^2+y^2=1$.
\end{enumerate}
These standard conics will let us categorize every marked conic, as they represent isomorphism classes of conics.

\begin{prop}
If $(\X,\Lref,\id)$ is any marked conic, then there is a projective transformation sending $(\X,\Lref,\id)$ to exactly one of the conics in the lemma. Which one is determined by how $\X$ and $\Lref$ intersect. 
\begin{enumerate}
    \item If $\X$ and $\Lref$ intersect at a double point, then $(\X,\Lref,\id)$ is isomorphic to the parabola. 
    \item If $\X$ and $\Lref$ intersect at two distinct points, then $(\X,\Lref,\id)$ is isomorphic to the hyperbola.
    \item If $\X$ and $\Lref$ do not intersect, then $(\X,\Lref,\id)$ is isomorphic to the circle.
\end{enumerate} 
\end{prop}
\begin{proof}
We can find a projective transformation which sends $\Lref$ to the line at infinity.\mycomment{
If $\Lref$ is the line $Ax+By+Cz=0$, and we set $c=(A,B,C)$ and $v=(x,y,z)$, then the map
$$
(x:y:z) \mapsto (a\cdot v : b\cdot v : c\cdot v)
$$
will send $\Lref$ to $\Linfty$ when the three vectors $a,b,c$ are linearly independent. Note that this is well defined as any scale factor to $(x:y:z)$ will scale $v$, which will just result in a scale in the final expression.} 
So we can then assume that we have a marked conic of the form $(\X,\Linfty,\id)$. How the conic $\X$ intersects $\Linfty$ determines if it is classically a parabola, hyperbola, or an ellipse. It is well known that there is some affine transformation which will send $\X$ to $\Par$, $\Hyp$, or $\Circ$. We just need to ensure that we can get $\id$ to be the point which is prescribed. For the $\Circ$ case, we can just use rotations to translate $\id$ to $(1,0)$ and we can similarly use Lorentz transformations on $\Hyp$ (we may have to mirror it first to ensure that $\id$ is on the correct component). The affine transformations which preserve the parabola have the form
$$
\begin{pmatrix}
x\\
x^2
\end{pmatrix}
\mapsto 
\begin{pmatrix}
a & 0\\
2ma & a^2
\end{pmatrix}\begin{pmatrix}
x\\
x^2
\end{pmatrix}
+ 
\begin{pmatrix}
m\\
m^2
\end{pmatrix}
=\begin{pmatrix}
ax+m\\
(ax+m)^2
\end{pmatrix}
$$
If $\id=(p,p^2)$, then taking $a=1$ and $m=-p$ will send $\id$ to $(0,0)$.
\end{proof}

It then follows that if these three standard conics all induce associative operations, then the same is true for every marked conic and Pascal's Theorem will follow. We  note for marked conics of the form $(\X,\Linfty,\id)$ that the binary operation is computed using the construction in the introduction using parallel lines. This allows for easy, direct computation of addition formulas in this case.
\begin{prop} 
\begin{enumerate}
    \item $(\mathcal{P},L_{\infty},O)$ induces a group on the affine part of $\mathcal{P}$ given by
    $$
    (a,a^2)\oplus (b,b^2) = (a+b,(a+b)^2)
    $$
    \item $(H,L_{\infty},O)$ induces a group on the affine part of $H$ given by
    $$
    (a,a^{-1})\oplus (b,b^{-1}) = (ab,a^{-1}b^{-1})
    $$
    \item  $(S^1,L_{\infty},O)$ induces a group on $S^1$ given by
    $$
    (a,b)\oplus (c,d) = (ac-bd,ad+bc)
    $$
\end{enumerate}
\end{prop}
\begin{proof}
See \hyperlink{https://www.jstor.org/stable/pdf/40378668.pdf}{ \citep{Author2Year2}}.

\end{proof}

These binary operations are isomorphic to addition, multiplication, and rotation, each of which is a group operation that is associative. This result then allows us to conclude, without using Pascal's Theorem, that:

\begin{theorem}
For any marked conic $(\X,\Lref,\id)$ the binary operation induced on $\XL$ turns it into a group with identity $\id$.
\end{theorem}
\begin{proof}
Let $\T$ be a projective transformation which sends $(\X,\Lref,\id)$ to one of the three standard marked conics. Through the direct computations above, the binary operation induced by these standard marked conics are either addition, multiplication, or rotation which all have group structure. Since $\T$ induces an isomorphism it follows that $(\X,\Lref,\id)$ induces a group structure on $\XL$.
\end{proof}

Pascal's Theorem then follows from the associativity of addition, multiplication, and rotation.

\bibliographystyle{abbrvnat}
\bibliography{main}

\begin{thebibliography}{7}
\providecommand{\natexlab}[1]{#1}
\providecommand{\url}[1]{\texttt{#1}}
\expandafter\ifx\csname urlstyle\endcsname\relax
  \providecommand{\doi}[1]{doi: #1}\else
  \providecommand{\doi}{doi: \begingroup \urlstyle{rm}\Url}\fi

\bibitem[{Acosta} and {Schlenker}(2020)]{Author5Year5}
M.~{Acosta} and J.~{Schlenker}.
\newblock {A hyperbolic proof of Pascal's Theorem}.
\newblock \emph{arXiv preprint arXiv:2012.14883}, Dec. 2020.

\bibitem[{Artzy}(1968)]{Author4Year4}
R.~{Artzy}.
\newblock {Pascal's theorem on an oval}.
\newblock \emph{The American Mathematical Monthly}, Feb. 1968.

\bibitem[{Caminata} and {Schaffler}(2019)]{Author6Year6}
A.~{Caminata} and L.~{Schaffler}.
\newblock {A Pascal's theorem for rational normal curves}.
\newblock \emph{arXiv preprint arXiv:1903.00460}, Mar. 2019.

\bibitem[{Lemmermeyer}(2003)]{AuthorYear}
F.~{Lemmermeyer}.
\newblock {Conics - A poor man's elliptic curves.}
\newblock \emph{arXiv preprint math/0311306}, 2003.

\bibitem[{Palej}(1997)]{Author3Year3}
M.~{Palej}.
\newblock {A simple proof for the theorems of Pascal and Pappus}.
\newblock \emph{Journal for Geometry and Graphics}, 1997.

\bibitem[{Shirali}(2009)]{Author2Year2}
S.~{Shirali}.
\newblock {Groups associated with conics}.
\newblock \emph{The Mathematical Gazette}, Mar. 2009.

\bibitem[{Tan}(1965)]{Author7Year7}
K.~{Tan}.
\newblock {Various proofs of Pascal's Theorem}.
\newblock \emph{Mathematics Magazine}, Jan. 1965.

\end{thebibliography}

\end{document}